\documentclass[a4paper]{amsart}

\usepackage[utf8]{inputenc}
\usepackage{newunicodechar}
\usepackage{lmodern}
\usepackage{textcomp}
\usepackage{hyperref}
\usepackage{xfrac}
\usepackage{tikz}
\usetikzlibrary{matrix,arrows,decorations.pathmorphing}
\usepackage{enumitem}
\usepackage{ dsfont }
\usepackage{ amssymb }

\newtheorem{satz}{Satz}
\newtheorem{thm}[satz]{Theorem}
\newtheorem{lem}[satz]{Lemma}
\newtheorem{cor}[satz]{Corollary}

\newtheorem{defi}[satz]{Definition}

\newtheorem{notation}[satz]{Notation}

\renewcommand{\O}{\mathcal{O}}

\renewcommand{\tilde}{\widetilde}
\renewcommand{\H}{\textup{H}}

\newcommand*{\defeq}{\mathrel{\vcenter{\baselineskip0.5ex \lineskiplimit0pt
			\hbox{\scriptsize.}\hbox{\scriptsize.}}}%
	=}

\newcommand\restr[2]{
	{\left.\kern-\nulldelimiterspace#1\vphantom{\big|}\right|_{#2}}
}

\newcommand{\Z}{\ensuremath{\mathds{Z}}}

\newcommand{\C}{\ensuremath{\mathds{C}}}
\newcommand{\Nn}{\ensuremath{\mathds{N}_{0}}}

\newcommand{\fL}{\ensuremath{\mathcal{L}}}
\newcommand{\fM}{\ensuremath{\mathcal{M}}}

\newcommand{\fU}{\ensuremath{\mathcal{U}}}
\newcommand{\fV}{\ensuremath{\mathcal{V}}}

\DeclareMathOperator{\ra}{\rightarrow}
\DeclareMathOperator{\lra}{\longrightarrow}

\DeclareMathOperator{\Ra}{\Rightarrow}

\DeclareMathOperator{\img}{img}

\DeclareMathOperator{\Pic}{Pic}
\DeclareMathOperator{\tor}{tor}

\let\enumerateO\enumerate
\let\endenumerateO\endenumerate
\renewenvironment{enumerate}{\enumerateO\setlength{\parskip}{0em}\setlength{\parindent}{1em}}{\endenumerateO}

\title{Low degree Hodge theory for klt varieties}
\date{\today}

\author{Martin Schwald}
\address{Martin Schwald, Mathematisches Institut, Albert-Ludwigs-Universität Freiburg, Eckerstraße 1, 79104 Freiburg im Breisgau, Germany}
\email{\href{mailto:martin.schwald@math.uni-freiburg.de}{martin.schwald@math.uni-freiburg.de}} \urladdr{\url{http://home.mathematik.uni-freiburg.de/schwald}}
\date{\today}

\thanks{The author gratefully acknowledges support by the DFG-Graduiertenkolleg GK1821
  ``Cohomological Methods in Geometry'' at the University of Freiburg}

\begin{document}

\begin{abstract}
	If $X$ is a complex projective variety with klt singularities, then the mixed Hodge structures on the first two singular cohomology groups are pure. We describe the pieces of the Hodge decomposition in terms of reflexive differential forms. Applications include a Lefschetz $(1,1)$ Theorem and a weak analogue of the Hodge-Riemann bilinear relations for klt varieties.
\end{abstract}

\maketitle

\tableofcontents

\section{Introduction}

\subsection{Main results} For a singular, complex projective variety $X$, Deligne constructed mixed Hodge structures on the singular cohomology groups $\H^p(X,\,\C)$, \cite{DeligneHodgeII}. Kirschner showed that the Hodge structure on $\H^2(X,\,\C)$ is pure when $X$ has rational singularities, \cite[Corollary~B.2.8]{Kir15}. When $X$ is klt, we describe the relevant pieces of the Hodge decomposition in terms of reflexive differential forms $\Omega^{[p]}_X$, the definition of which we recall in \ref{forms}.

\begin{thm}[Structure of $\H^1(X,\,\C)$ and $\H^2(X,\,\C)$, see Theorem~\ref{hdec} for details]\label{mainthm}
  If $X$ is a complex projective variety with at most klt
  singularities, then for $p=1,2$ the Hodge structure $\H^p(X,\,\C)$ is pure of
  weight $p$ and there are canonical isomorphisms
  $$
  \kappa_{p0} : \H^{p,0}(X) \to \H^0(X,\,\Omega^{[p]}_X)  \quad\text{and}\quad \kappa_{0p} : \H^{0,p}(X) \to \H^p(X,\,\O_X).
  $$
\end{thm}

We give two applications of this result. In Theorem~\ref{lef} we prove a Lefschetz $(1,1)$ Theorem for the $\H^{1,1}(X)$ part. This means that the integral points in $\H^{1,1}(X)$ are precisely the first Chern classes of line bundles on $X$. Furthermore, we prove in Corollary~\ref{bilrel} an analogue of the Hodge-Riemann bilinear relations for klt varieties, in the special case where only forms of degree one and two are involved.

\subsection{Further applications to symplectic varieties}

In a further paper we apply the results found here to holomorphic symplectic varieties, which automatically have canonical singularities. There we show that the generalized Beauville-Bogomolov form on an irreducible symplectic variety $X$ has the index ${(3,0,b_2(X)-3)}$, like in the smooth case. As a corollary, every non-trivial fibration of an irreducible symplectic variety $X$ onto a normal, complex projective variety $B$ is up to Stein factorization a Lagrangian fibration with $\dim B=\frac{1}{2}\dim X$ and Picard number $\rho(B)=1$.

\subsection{Acknowledgement}

The author wants to thank his supervisor Prof.\,Kebekus for his advice and many fruitful discussions.

\section{Differential forms and Extension Theorem}

\subsection{Holomorphic, torsion-free and reflexive differential forms}
\label{forms}
Let $X$ be a complex variety, not necessarily irreducible. We denote the sheaf of \emph{Kähler differentials} by $\Omega^1_X$. The sheaves of \emph{holomorphic differential forms} are the exterior powers ${\Omega^p_X\defeq\bigwedge^p\Omega^1_X}$ with the convention ${\Omega^0_X\defeq\O_X}$.

Denoting with $\fM_X$ the sheaf of rational functions on $X$, we call the kernel of the natural map ${\Omega^p_X\ra\Omega^p_X\otimes_{\O_X}\fM_X}$ the \emph{torsion subsheaf}, $\tor\Omega^p_X$. For the quotient we use the notation ${\check{\Omega}^p_X\defeq\sfrac{\Omega^p_X}{\tor\Omega^p_X}}$. By slight abuse of language, we call $\check{\Omega}^p_X$ the sheaf of \emph{torsion-free differential forms}, cf.\,{\cite[II]{MR3084424}}.

\emph{Reflexive differential forms} are holomorphic $p$-forms on the smooth locus $X_{reg}$ of a variety $X$. We notate the associated sheaf on $X$ as $\Omega^{[p]}_X$ and get
\begin{align*}
	\Omega^{[p]}_X&=(\Omega^p_X)^{\ast\ast}=i_*\Omega^p_{X_{reg}},\\ \H^0(X,\,\Omega^{[p]}_X)&\cong\H^0(X_{reg},\,\Omega_{X_{reg}}^p),
\end{align*}
where $i\colon X_{reg}\hookrightarrow X$ is the inclusion. See
\cite{MR597077} for a reference on \emph{reflexive sheaves} and \cite[I--III]{GKKP11}, \cite[I]{Kebekus2016} for more information on reflexive differential forms.

\subsection{Extension Theorem for reflexive differential forms}
\label{extensionsection}
The extension theorem for differential forms \cite[Theorem~1.4]{GKKP11} shows that on a klt variety $X$ any reflexive form ${\alpha\in\H^0(X,\,\Omega^{[p]}_X)}$  can be \emph{extended} to any resolution of singularities. This allows us to construct pullbacks of reflexive forms along more general morphisms. 

\begin{thm}[Pullbacks of reflexive forms, \protect{\cite[Theorem~1.3]{MR3084424}}]\label{extension}
  Let $f\colon Y\ra X$ be a morphism between normal, complex projective
  varieties. Assume that there is a Weil divisor $D$ on $X$, such that $(X,D)$ is klt. Then there is a natural pullback morphism $f^*\Omega^{[p]}_X\ra\Omega^{[p]}_Y$, consistent with the natural pullback of Kähler differentials on $X_{reg}$.\qed
\end{thm}

 The consistency with the natural pullback of Kähler differentials is made precise in \cite[Theorem~5.2]{MR3084424}. We will only consider the case when $f\colon Y\ra X$ is surjective, for example a birational morphism. Then this consistency means that the pullback morphism $f^*\Omega^{[p]}_X\ra\Omega^{[p]}_Y$ agrees with the usual pullback of Kähler differentials wherever $X$ and $Y$ are smooth. In other words, for every $\alpha\in\H^0(X,\,\Omega^{[p]}_X)$ there is an $\tilde{\alpha}\in\H^0(Y,\,\Omega^{[p]}_Y)$ that coincides with $f^*(\restr{\alpha}{X_{reg}})$, seen as a Kähler differential on $Y_{reg}\cap f^{-1}(X_{reg})$. We call $\tilde{\alpha}$ an \emph{extension} of $\alpha$ to $Y$. Pulling back reflexive forms is contravariantly functorial with respect to $f$. 

\section{Hodge Theory for klt singularities}
\label{hodgetheory}

\subsection{Cohomology of fibers}

As a preparation for the proof of Theorem~\ref{mainthm}, we mention a consequence of a triangulation theorem of {\L}ojasiewitcz. It is frequently used in the literature, but the author does not know of a convenient reference. Moreover, we will need results about differential forms on the fibers of a resolution of a variety with klt singularities.

\begin{thm}
	\label{loj}
	Let $\nu\colon\tilde{X}\ra X$ be a resolution of a compact, complex variety $X$. Then every point $p\in X$ has a basis of analytical neighborhoods $(U_i)_i$, such that the fiber $E\defeq\nu^{-1}(p)$ is a deformation retract of each $\nu^{-1}(U_i)$. Thus for any $k\in\Z$ the restrictions ${\H^k\big(\nu^{-1}(U_i),\,\Z\big)\stackrel{\sim}{\lra}\H^k(E,\,\Z)}$ are isomorphisms and we get ${(R^k\nu_*\Z)_p\cong\H^k(E,\,\Z)}$.
\end{thm}
\begin{proof}
	Consider any semi-analytic neighborhood $U$ of $p$ and $V\defeq\nu^{-1}(U)$. {\L}ojasiewitcz's result \cite[Theorem~4]{Loj} applied to the semi-analytic subsets $E$ and $V$ of the smooth manifold $\tilde{X}$ gives a locally finite triangulation of $V$, such that $E$ is a subcomplex. Then \cite[Proposition~A.5]{Hatcher} gives a deformation retraction from any sufficiently small analytical neighbourhood $V'$ of $E$ to the fiber $E$. As $\tilde{X}$ is compact, we can choose a basis of these neighbourhoods of the form $\nu^{-1}(U_i)$ with the $(U_i)_i$ a basis of analytical neighborhoods of $p$.
	
	As deformation retractions are homotopies, we get ${\H^k\big(\nu^{-1}(U_i),\,\Z\big)\stackrel{\sim}{\lra}\H^k(E,\,\Z)}$ for any $k\in\Z$. Recall from \cite[Proposition III.8.1]{Har77} that the stalks $(R^k\nu_*\Z)_p$ can be calculated as $\lim_U \H^k(\nu^{-1}(U),\,\Z)$, where the limit is taken over all analytic neighborhoods $U$ of $p$, which gives ${(R^k\nu_*\Z)_p\cong\H^k(E,\,\Z)}$.
\end{proof}

\begin{thm}[\protect{Torsion-free differential forms on fibers, \cite[Theorem~4.1]{MR3084424}, \cite{HMcK07}}]
	\label{torfree}
	Let $X$ be a complex variety with at most klt singularities. We take any resolution $\nu\colon\tilde{X}\ra X$. Let $E$ be the reduction of a fiber of $\nu$. Then there are no global torsion-free differential forms on $E$. In other words ${\H^0(E,\,\check{\Omega}^p_E)=0}$ for $p>0$.\qed
\end{thm}
 Previously, Namikawa proved in \cite[Lemma~1.2]{NamikawaDeformationTheory} a similar statement for $X$ with rational singularities, but only with the assumption that $E$ is an snc fiber. The proof of Theorem~\ref{torfree} uses that by \cite{HMcK07} every fiber $E$ is rationally chain connected. Hence $E$ supports only torsion differential forms.

\begin{lem}
	\label{formsonfiber}
	Let $X$ be a complex projective klt variety. Consider a resolution $\nu\colon\tilde{X}\ra X$ and $E$ the reduction of any fiber of $\nu$. Assume that $E$ has only smooth components. Then for every $p>0$ the restriction $F^p\H^p(\tilde{X},\,\C)\ra F^p\H^p(E,\,\C)$ is the zero map, where $F^{\bullet}$ denotes the Hodge filtration.
\end{lem}
\begin{proof}
	Let ${E=\bigcup_{i=1}^{r} E_i}$ be the decomposition of $E$ into its smooth, irreducible components and denote ${E^k\defeq\bigcup_{i=1}^kE_i}$ for all $k$. The restriction of singular cohomology classes $\H^p(\tilde{X},\,\C)\ra\H^p(E^k,\,\C)$ is a morphism of mixed Hodge structures, which restricts to a morphism between the $F^p$-parts. To prove the Lemma we inductively show that $F^p\H^p(\tilde{X},\,\C)\ra F^p\H^p(E^k,\,\C)$ is the zero map for any $k$.	
	
	\subsubsection*{Proof that $F^p\H^p(\tilde{X},\,\C)\ra F^p\H^p(E_i,\,\C)$ is zero for all $i$.}
	
	 As $\tilde{X}$ and all components $E_i$ are smooth, we have $F^p\H^p(\tilde{X},\,\C)=\H^0(\tilde{X},\,\Omega^p_{\tilde{X}})$ and ${F^p\H^p(E_i,\,\C)=\H^0(E_i,\,\Omega^p_{E_i})}$. Furthermore, all holomorphic $p$-forms on $\tilde{X}$ and on the $E_i$ are torsion-free. Hence we get by \cite[Proposition~A.6]{MR3084424} the following commutative diagram.
	
	\begin{center} 
		\begin{tikzpicture}[scale=1.5]
		\node (A) at (0,0) {$\H^0(\tilde{X},\,\check{\Omega}^p_{\tilde{X}})$};
		\node (B) at (2,0) {$\H^0(E,\,\check{\Omega}^p_{E})$};
		\node (C) at (4,0) {$\H^0(E_i,\,\check{\Omega}^p_{E_i})$};
		\node (D) at (0,1) {$\H^0(\tilde{X},\,\Omega^p_{\tilde{X}})$};
		\node (E) at (2,1) {$\H^0(E,\,\Omega^p_E)$};
		\node (F) at (4,1) {$\H^0(E_i,\,\Omega^p_{E_i})$};
		
		\path[->,font=\scriptsize]
		(D) edge node[left]{$\wr$} (A)
		(E) edge (B)
		(F) edge node[left]{$\wr$} (C)
		(A) edge (B)
		(B) edge (C)
		(D) edge (E)
		(E) edge (F);
		\end{tikzpicture}
	\end{center}
	Here, the horizontal maps are restriction maps and the vertical maps quotient out torsion diffential forms. By Theorem~\ref{torfree} we have $\H^0(E,\,\check{\Omega}^p_{E})=0$, so ${\H^0(\tilde{X},\,\Omega^p_{\tilde{X}})\ra\H^0(E_i,\,\Omega^p_{E_i})}$ is the zero map.
	
	\subsubsection*{Proof that $F^p\H^p(\tilde{X},\,\C)\ra F^p\H^p(E^{k+1},\,\C)$ is zero for any $k\in\{1,\ldots, r-1\}$.}	
	 We use the standard Mayer-Vietoris sequence, \cite[page~203]{Hatcher}.
	\[
	\cdots\ra\H^{p-1}(E^k\cap E_{k+1},\C)\ra \H^p(\underbrace{E^k\cup E_{k+1}}_{=E^{k+1}},\C)\ra \H^p(E^k,\,\C)\oplus\H^p(E_{k+1},\,\C)\ra\cdots
	\]
	This is an exact sequence of mixed Hodge structures, hence we can restrict it to the $F^p$-part of the Hodge filtration. We have ${F^p\H^{p-1}(E^k\cap E_{k+1},\C)=0}$ by the definition of the mixed Hodge structure on a complex projective variety. This gives us the inclusion ${F^p\H^p(E^{k+1},\,\C)\hookrightarrow F^p\H^p(E^k,\,\C)\oplus F^p\H^p(E_{k+1},\,\C)}$, which is, up to sign, the sum of the two restrictions ${F^p\H^p(E^{k+1},\,\C)\ra F^p\H^p(E^k,\,\C)}$ and ${F^p\H^p(E^{k+1},\,\C)\ra F^p\H^p(E_{k+1},\,\C)}$. Hence we get the following commutative diagram.
	\begin{center}
	\begin{tikzpicture}[scale=1.5]
	\node (A) at (0,0){$F^p\H^p(\tilde{X},\,\C)$};
	\node (B) at (2,0){$F^p\H^p(E^{k+1})$};
	\node (C) at (5,0){$F^p\H^p(E^k,\,\C)\oplus F^p\H^p(E_{k+1},\,\C)$};
	\node (D) at (5,1){$F^p\H^p(E^k,\,\C)$};
	\node (E) at (5,-1){$F^p\H^p(E_{k+1},\,\C)$};
	\path[->,font=\scriptsize]
	(A) edge (B)
	(C) edge (D)
	(C) edge (E)
	(A) edge node[above]{0} (D)
	(A) edge node[below]{0} (E);
	
	\path[right hook->,font=\scriptsize]
	(B) edge (C);
	\end{tikzpicture}
	\end{center}
	As the diagonal maps are zero by the induction hypothesis, the map $F^p\H^p(\tilde{X},\,\C)\ra F^p\H^p(E^{k+1})$ is also the zero map. This completes the proof.
\end{proof}

%\begin{rem}
	%Although the restriction map $F^p\H^p(\tilde{X},\,\C)\ra F^p\H^p(E,\,\C)$ is zero, the fiber $E$ might carry non-trivial differential forms. We can take for example $X=\P^3$ and for $\tilde{X}$ the result of first blowing up of a point $p\in X$ and an elliptic curve $C$ on the exceptional divisor. The fiber $E$ over $p$ has then two irreducible components, $\P^2$ and a $\P^1$-bundle over $C$. We can now take a non-trivial one-form on $C$, extend it constantly along the $\P^1$-fibers and extend it by zero along the $\P^2$
%\end{rem}

To make use of the preceding Lemma~\ref{formsonfiber}, we need to make sure that we can always find a resolution fulfilling the needed assumption on the fibers. This is a variation of the frequently used strong resolution of singularities from \cite[Theorem~3.35]{Kollar07}.

\begin{lem}[Resolution of singularities with nice fibers]
\label{niceresolution}
Let $X$ be a complex projective variety. Then there is a resolution of singularities $\nu\colon\tilde{X}\ra X$, given as a sequence of blowups of smooth centers over the singular locus $X_{sing}$, such that the reduction of every fiber of $\nu$ has only smooth components.
\end{lem}
\begin{proof}
Let $f\colon Y\ra X$ be any resolution of singularities. We write $Z(f)\subset X$ for the Zariski closure of the locus of points $x\in X$, for which their reduced fiber ${f^{-1}(x)_{red}}$ has a singular component. We notate the preimage as ${E(f)\defeq f^{-1}\big(Z(f)\big)}$.

Starting with an arbitrary resolution $f\colon Y\ra X$, we want to modify it to shrink $Z(f)$, until we eventually obtain $Z(f)=\varnothing$. For this we choose by \cite[Theorem~3.35]{Kollar07} an embedded resolution $g\colon Y'\ra Y$ of $E(f)\subset Y$ by blowing up smooth centers over $E(f)$. This is a birational morphism $g$ from a smooth variety $Y'$, such that the preimage $E\defeq g^{-1}\big(E(f)\big)$ is an snc divisor and $g$ is an isomorphism over $Y\setminus E(f)$. 

We consider the resolution $f'\defeq f\circ g$ of $X$. The fibers of $f'$ over points $x\in Z(f)$ are contained in the subvariety $E\subset Y'$. All irreducible components of $E$ are smooth by construction of $g$ and we denote them as $E_i$. Morphisms from smooth complex varieties are by \cite[Corollary~10.7]{Har77} generically smooth. Thus we find open subsets $U_i\subset Z(f)$ such that $\restr{f'}{E_i}$ is smooth over $U_i$ and the intersection $\bigcap_i U_i$ is a non-empty open subset of $Z(f)$. Therefore the reduction of every fiber of $\restr{f'}{E}$ over $\bigcap_i U_i$ has only smooth components. The reductions of fibers of $f'$ over points $x\in X\setminus Z(f)$ have smooth components by the assumptions on $f$ and $g$. Hence $Z(f')$ is a proper closed subset of $Z(f)$.

We can replace the resolution $(Y,f)$ by $(Y',f')$ and repeat this step, until we eventually obtain $Z(f)=\varnothing$. This terminates because the Zariski topology on $X$ is Noetherian. We conclude that there is a resolution $\nu\colon\tilde{X}\ra X$ with $Z(\nu)=\varnothing$.
\end{proof}

\subsection{The Hodge decompositions of \texorpdfstring{$\H^1(X,\,\C)$ and $\H^2(X,\,\C)$}{the first and second singular cohomology groups}}

In this section we prove Theorem~\ref{mainthm} in a similar way to \cite[B.2.4--B.2.9]{Kir15}. We take a resolution $\nu\colon\tilde{X}\ra X$ and calculate the low degree terms of the Leray spectral sequence for $\nu$ and the constant sheaf $\C$ on $\tilde{X}$. More precisely, we prove in Theorem~\ref{hdec} a more detailed version of Theorem~\ref{mainthm}, which contains also more information on the maps $\nu^*\colon\H^p(X,\,\C)\ra\H^p(\tilde{X},\,\C)$ for $p=1,2$.

\begin{thm}[Structure of $\H^1(X,\,\C)$ and $\H^2(X,\,\C)$]
\label{hdec}
  Let $X$ be a complex, projective variety with at most klt singularities and
  let $\nu\colon\tilde{X}\ra X$ be a resolution of singularities. Then for
  ${p=1,2}$ the pullback $\nu^*\colon\H^p(X,\,\C)\ra\H^p(\tilde{X},\,\C)$ is a
  monomorphism of mixed Hodge structures and the mixed Hodge structure on $\H^p(X,\,\C)$ is pure of weight $p$. Hence $\nu^*$ induces injective pullback maps
  \[
  \nu^*_{ab}: \H^{a,b}(X) \to \H^{a,b}(\tilde X)
  \quad\text{for } a+b=p,
  \]
  where the morphisms $\nu^*_{p0}$ and $\nu^*_{0p}$ are isomorphisms. There are
  canonical isomorphisms $\kappa_{p0}\colon\H^{p,0}(X)\stackrel{\sim}{\lra}\H^0(X,\, \Omega^{[p]}_X)$ and $\kappa_{0p}\colon\H^{0,p}(X)\stackrel{\sim}{\lra}\H^p(X,\, \O_X)$ that commute with the natural pullback morphisms and the Dolbeault isomorphisms $\delta_{p0},\delta_{0p}$.
	\begin{center}
		\begin{tikzpicture}[scale=1.5]
		\node (A) at (0,0) {$\H^{p,0}(X)$};
		\node (B) at (2,0) {$\H^0(X,\,\Omega^{[p]}_X)$};
		\node (C) at (0,1) {$\H^{p,0}(\tilde{X})$};
		\node (D) at (2,1) {$\H^0(\tilde{X},\,\Omega^p_{\tilde{X}})$};
		
		\path[->,font=\scriptsize]
		(A) edge node[above]{$\sim$} node[below]{$\kappa_{p0}$} (B)
		(C) edge node[above]{$\sim$} node[below]{$\delta_{p0}$} (D)
		(A) edge node[left]{$\nu^*_{p0}$} node[right]{$\wr$} (C)
		(B) edge node[left]{$ext$} node[right]{$\wr$} (D);
		\end{tikzpicture}
		\quad\quad\quad\quad
		\begin{tikzpicture}[scale=1.5]
		\node (A) at (0,0) {$\H^{0,p}(X)$};
		\node (B) at (2,0) {$\H^p(X,\,\O_X)$};
		\node (C) at (0,1) {$\H^{0,p}(\tilde{X})$};
		\node (D) at (2,1) {$\H^p(\tilde{X},\,\O_{\tilde{X}})$};
	
		\path[->,font=\scriptsize]
		(A) edge node[above]{$\sim$} node[below]{$\kappa_{0p}$} (B)
		(C) edge node[above]{$\sim$} node[below]{$\delta_{0p}$} (D)
		(A) edge node[left]{$\nu^*_{0p}$} node[right]{$\wr$} (C)
		(B) edge node[left]{$\nu^*$} node[right]{$\wr$} (D);
		\end{tikzpicture}
	\end{center}
\end{thm}
\begin{proof}
	We proceed in five steps:
	\begin{enumerate}
		\item\label{inj} Calculate the low degree terms of the Leray spectral sequence for $\nu\colon\tilde{X}\ra X$ and the constant sheaf $\C$ on $\tilde{X}$.
		\item\label{pureness} Prove the bijectivity of ${\nu^*\colon\H^1(X,\,\C)\ra\H^1(\tilde{X},\,\C)}$, the injectivity of ${\nu^*\colon\H^2(X,\,\C)\ra\H^2(\tilde{X},\,\C)}$ and the purity of the Hodge structures on $\H^1(X,\,\C)$ and $\H^2(X,\,\C)$.
		\item\label{surj} Prove the surjectivity of the restrictions $\nu^*_{20}$ and $\nu^*_{02}$ of the map  ${\nu^*\colon\H^2(X,\,\C)\ra\H^2(\tilde{X},\,\C)}$, assuming that all reductions of fibers of $\nu$ have only smooth components.
		\item\label{nicefibers} Prove the surjectivity of $\nu^*_{20}$ and $\nu^*_{02}$ in general.
		\item\label{H20H02} Calculate $\H^{p,0}(X),\,\H^{0,p}(X)$ and construct the canonical isomorphisms $\kappa_{p0},\,\kappa_{0p}$ for $p=1,2$.
	\end{enumerate}
	
	\subsubsection*{Step \eqref{inj}.} The Leray spectral sequence of the constant sheaf $\C$ on $\tilde{X}$ is
	\[
	E_2^{pq}\defeq\H^p(X,\,R^q\nu_*\C)\Ra\H^{p+q}(\tilde{X},\,\C).
	\]
	Pushing forward the exponential exact sequence gives the exact sequence
	\begin{center}
	\begin{tikzpicture}[scale=1]
	\node (A) at (0,0) {$\nu_*\O_{\tilde{X}}$};
	\node (B) at (2,0) {$\nu_*\O_{\tilde{X}}^{\times}$};
	\node (C) at (4,0) {$R^1\nu_*\Z$};
	\node (D) at (6.5,0) {$R^1\nu^*\O_{\tilde{X}}=0.$};
	\path[overlay,->>, font=\scriptsize]
	(A) edge node[above]{$\exp$} (B);
	\path[overlay,->, font=\scriptsize]
	(B) edge node[above]{$0$} (C);
	\path[overlay,right hook->, font=\scriptsize]	
	(C) edge (D);
	\end{tikzpicture}
	\end{center}	
	As the fibers are connected, the first map is induced by the exponential map $\O_X\ra\O_X^{\times}$, which is surjective, so the second map is the zero map. This shows that $R^1\nu_*\Z$ injects into $R^1\nu_*\O_{\tilde{X}}$ and hence has to vanish, as klt singularities are rational, \cite[Theorem~5.22]{KM98}. 	Therefore $R^1\nu_*\C=0$ and the first terms of the second page $(E_2,d_2)$ are:
	\begin{center}
			\begin{tikzpicture}[scale=2.0]
			\node (A) at (0,0) {$E^{00}_2$};
			\node (B) at (1,0) {$E^{10}_2$};
			\node (C) at (2,0) {$E^{20}_2$};
			\node (D) at (0,0.5) {$E^{01}_2$};
			\node (E) at (1,0.5) {$E^{11}_2$};
			\node (F) at (2,0.5) {$E^{21}_2$};
			\node (G) at (0,1) {$E^{02}_2$};			
			
			\path[->,font=\scriptsize]
			(D) edge node[pos=0.2,above]{$d_2$} (C)
			(G) edge node[pos=0.2,above]{$d_2$} (F);
			\end{tikzpicture}
			\quad\quad\quad\quad
			\begin{tikzpicture}[scale=2.0]
			\node (A) at (0,0) {$\H^0(X,\,\C)$};
			\node (B) at (1,0) {$\H^1(X,\,\C)$};
			\node (C) at (2,0) {$\H^2(X,\,\C)$};
			\node (D) at (0,0.5) {$0$};
			\node (E) at (1,0.5) {$0$};
			\node (F) at (2,0.5) {$0$};
			\node (G) at (0,1) {$\H^0(X,\,R^2\nu_*\C)$};			
						
			\path[->,font=\scriptsize]
			(D) edge node[pos=0.2,above]{$d_2$} (C)
			(G) edge node[pos=0.2,above]{$d_2$} (F);
			\end{tikzpicture}
\end{center}
	The entries $(0,0),\,(1,0),\,(2,0)$ and the second row degenerate at the second page. We see ${E^{02}_3\cong E^{02}_2=\H^0(X,\,R^2\nu_*\C)}$ and the entry $(0,2)$ degenerates at the fourth page with ${E^{02}_4=\ker\big(d_3\colon E^{02}_2\ra \H^3(X,\,\C)\big)}\subset\H^0(X,\,R^2\nu_*\C)$, which shows ${\H^2(\tilde{X},\,\C)\cong \H^2(X,\,\C)\oplus E^{02}_4}$. This gives the following exact sequence.	
	\begin{equation}
	\begin{gathered}
	\label{lerayseq}
	\begin{tikzpicture}
	\matrix (m) [
	matrix of math nodes,
	row sep=3em,
	column sep=2.5em,
	text height=1.5ex, text depth=0.25ex
	]
	{ 0 & \H^1(X,\,\C) & \H^1(\tilde{X},\,\C) & \H^0(X,\,R^1\nu_*\C) \\
		& \H^2(X,\,\C) & \H^2(\tilde{X},\,\C) & \H^0(X,\,R^2\nu_*\C) \\
              };
              
              \path[overlay,->, font=\scriptsize]
              (m-1-1) edge (m-1-2)
              (m-1-2) edge node[above]{$\nu^*$} (m-1-3)
              (m-1-3) edge node[above]{$\pi_{01}$} (m-1-4)
              (m-1-4) edge[out=355,in=175] node[above]{$d_2$} (m-2-2)
              (m-2-2) edge node[above]{$\nu^*$} (m-2-3)
              (m-2-3) edge node[above]{$\pi_{02}$} (m-2-4);
            \end{tikzpicture}
	\end{gathered}
	\end{equation}
		
	\subsubsection*{Step \eqref{pureness}.} As $\H^0(X,\,R^1\nu_*\C)$ vanishes, sequence~\eqref{lerayseq} shows that ${\nu^*\colon\H^1(X,\,\C)\ra\H^1(\tilde{X},\,\C)}$ is an isomorphism and ${\nu^*\colon\H^2(X,\,\C)\ra\H^2(\tilde{X},\,\C)}$ is injective. They are also morphisms of mixed Hodge structures, which by Deligne are always strict with respect to both filtrations, cf.\,\cite[Lemma~1.13]{GS75}.
	Thus both maps $\nu^*$ induce also in any weight monomorphisms of pure Hodge structures. As the mixed Hodge structure on $\H^p(\tilde{X},\,\C)$ is pure of weight $p$, the same is true for $\H^p(X,\,\C)$ for $p=1,2$.
	
	\subsubsection*{Step \eqref{surj}.}
	We want to prove that the restriction $\nu^*_{20}\colon\H^{2,0}(X)\ra\H^{2,0}(\tilde{X})$ of ${\nu^*\colon\H^2(X,\,\C)\ra\H^2(\tilde{X},\,\C)}$ is surjective. By the exact sequence~\eqref{lerayseq} it is enough to show that $\pi_{02}$ is the zero map on the subspace $\H^{2,0}(\tilde{X})=F^2\H^2(\tilde{X},\,\C)$. First of all, to be able to apply Lemma~\ref{formsonfiber} later, we assume that all reductions of fibers of $\nu$ have only smooth components.
	
	To calculate $\pi_{02}$, we choose by Theorem~\ref{loj} open coverings ${\fU=(U_i)_i}$ of $X$ and ${\fV\defeq (V_i)_i\defeq\big(\nu^{-1}(U_i)\big)_i}$ of $\tilde{X}$, such that for each $i$ we have ${\H^0(U_i,\,R^2\nu_*\C)\cong\H^2(E_i,\,\C)}$ for a fiber $E_i$ of $\nu$ over a point in $U_i$. By choosing the $U_i$ to be contractible and an acyclic covering of $X$, we can calculate $\pi_{02}$ via \v{C}ech cohomology as the following composition.
	\begin{center}
	\begin{tikzpicture}[scale=1.5]
	\node (A) at (1.5,1) {$\H^2(\tilde{X},\,\C)$};
	\node (B) at (1.5,0) {$\prod_i\H^2(E_i,\,\C)$};
	\node (C) at (4,0) {$\prod_i\H^0(U_i,\,R^2\nu_*\C)$};
	\node (D) at (6.5,0) {$\check{\H}^0(X,\,\fU,\,R^2\nu_*\C)$};
	\node (E) at (6.5,1) {$\H^0(X,\,R^2\nu_*\C)$};
	\node (G) at (0,1) {$\gamma$};	
	\node (H) at (0,0) {$(\restr{\gamma}{E_i})_i$};		
	\path[->,font=\scriptsize]
	(A) edge (B)
	(B) edge node[above]{$\sim$} (C)
	(D) edge node[right]{$\wr$} (E)
	(A) edge node[above]{$\pi_{02}$} (E);
	\path[->>,font=\scriptsize]		
	(C) edge (D);
	\path[|->,font=\scriptsize]
	(G) edge (H);
	\end{tikzpicture}
	\end{center}
	By assumption, all components of the reduced fibers $(E_i)_{red}$ are smooth. As the restrictions $\H^2(\tilde{X},\,\C)\ra\H^2(E_i,\,\C)$ are morphisms of Hodge structures, they restrict to maps $F^2\H^2(\tilde{X},\,\C)\ra F^2\H^2(E_i,\,\C)$ that vanish by Lemma~\ref{formsonfiber}. We see that the restriction of $\pi_{02}$ to $\H^{2,0}(\tilde{X})$ factors through the zero map and hence is the zero map itself. We conclude that $\nu^*_{20}$ is surjective. By Hodge symmetry the same holds for $\nu^*_{02}$.
	
\subsubsection*{Step \eqref{nicefibers}.}	
	We use Lemma~\ref{niceresolution} to blow up smooth centers over $\tilde{X}$, until we get a $\tilde{\nu}\colon X'\ra\tilde{X}$, where $X'$ is smooth and the reductions of all fibers of $\mu\defeq\nu\circ\tilde{\nu}\colon X'\ra X$ have only smooth components.
	
	By part~\eqref{inj} pulling back with $\mu,\nu,\tilde{\nu}$ on the $\H^{2,0}$ parts gives injections
	\begin{center}
		\begin{tikzpicture}[scale=1.5]
		\node (A) at (0,0) {$\mu^*_{20}\colon\H^{2,0}(X)$};
		\node (B) at (2,0) {$\H^{2,0}(\tilde{X})$};
		\node (C) at (4,0) {$\H^{2,0}(X').$};
		\path[overlay,right hook->, font=\scriptsize]
		(A) edge node[above]{$\nu_{20}^*$} (B)	
		(B) edge node[above]{$\tilde{\nu}_{20}^*$} (C);
		\end{tikzpicture}
	\end{center}
The map $\mu^*_{20}$ is surjective by part~\eqref{surj}, hence also $\nu_{20}^*$. The surjectivity of $\nu_ {02}^*$ follows analogously.
	
\subsubsection*{Step \eqref{H20H02}.}
	As $X$ is normal, we have $\nu_*\O_{\tilde{X}}\cong\O_X$ by Zariski's Main Theorem, \cite[Corollary~11.4]{Har77}. As klt singularities are rational, the Leray spectral sequence for ${\nu\colon\tilde{X}\ra X}$ and $\O_{\tilde{X}}$ dege\-ne\-rates on $E_2$ with zeros away from the row ${E_2^{p0}=\H^p(X,\,\nu_*\O_{\tilde{X}})}$. This gives us ${\nu^*\colon\H^p(X,\,\nu_*\O_{\tilde{X}})\stackrel{\sim}{\lra}\H^p(\tilde{X},\,\O_{\tilde{X}})}$ for all $p$, cf.\,\cite[III~ex.\,8.1]{Har77}. We also have the isomorphisms ${ext\colon\H^0(X,\,\Omega^{[p]}_X)\stackrel{\sim}{\lra}\H^0(\tilde{X},\,\Omega^p_{\tilde{X}})}$ from the extension theorem and the Dolbeault isomorphisms ${\delta_{p0}\colon\H^0(\tilde{X},\,\Omega^p_{\tilde{X}})\stackrel{\sim}{\lra}\H^{p,0}(\tilde{X})}$, ${\delta_{0p}\colon\H^0(\tilde{X},\,\O_{\tilde{X}})\stackrel{\sim}{\lra}\H^{0,p}(\tilde{X})}$. Hence we can define for $p=1,2$ canonical isomorphisms $\kappa_{p0}$ and $\kappa_{0p}$ as compositions of these isomorphisms.
	\begin{center}
		\begin{tikzpicture}[scale=1.5]
		\node (A) at (0,0.5) {$\kappa_{p0}\colon\H^0(X,\,\Omega^{[p]}_X)$};
		\node (B) at (2.4,0.5) {$\H^0(\tilde{X},\,\Omega^p_{\tilde{X}})$};
		\node (C) at (4.4,0.5) {$\H^{p,0}(\tilde{X})$};
		\node (D) at (6.4,0.5) {$\H^{p,0}(X)$};
		\node (E) at (0,0) {$\kappa_{0p}\colon\H^p(X,\,\O_X)$};
		\node (F) at (2.4,0) {$\H^p(\tilde{X},\,\O_{\tilde{X}})$};
		\node (G) at (4.4,0) {$\H^{0,p}(\tilde{X})$};			
		\node (H) at (6.4,0) {$\H^{0,p}(X)$};	
		\path[->,font=\scriptsize]
		(A) edge node[above]{$ext$} (B)
		(B) edge node[above]{$\delta_{p0}$} (C)
		(C) edge node[above]{$(\nu_{p0}^*)^{-1}$} (D)
		(E) edge node[above]{$\nu^*$} (F)
		(F) edge node[above]{$\delta_{0p}$} (G)
		(G) edge node[above]{$(\nu_{0p}^*)^{-1}$} (H);
		\end{tikzpicture}
	\end{center}\!\!\!\!\!\!\!
\end{proof}	

\subsection{Lefschetz (1,1) Theorem for klt singularities}

 The classical Lefschetz $(1,1)$ Theorem is usually stated for smooth, complex projective varieties, see for example \cite[page~163--164]{GH}. In \cite[Theorem 1.1]{BS00} there is a Lefschetz $(1,1)$ Theorem for any normal, complex projective variety. We give a short proof of the theorem for klt varieties, where we do not need to work with mixed Hodge structures by Theorem~\ref{hdec}.
 
\begin{thm}[Lefschetz~(1,1)~Theorem for klt singularities]
 	\label{lef} Let $X$ be a complex projective variety with at most klt singularities. The integral points in $\H^{1,1}(X)$ are exactly the first Chern classes of line bundles on $X$. In other words, for the canonical map $\iota\colon\H^2(X,\,\Z)\ra\H^2(X,\,\C)$ we get
 	\[
 	\H^{1,1}(X)\cap\img\iota=\img\big(c_1\colon\Pic(X)\ra\H^2(X,\,\Z)\big).
 	\]
\end{thm}
\begin{proof}
	We take a resolution $\nu\colon\tilde{X}\ra X$ and consider the cano\-ni\-cal maps $\iota\colon\H^2(X,\,\Z)\ra\H^2(X,\,\C)$, $\tilde{\iota}\colon\H^2(\tilde{X},\,\Z)\ra\H^2(\tilde{X},\,\C)$ that extend the scalars. As $\nu^*\colon\H^2(X,\,\C)\ra\H^2(\tilde{X},\,\C)$ preserves the Hodge decomposition, we get for any class ${\alpha\in\H^{1,1}(X)\cap\img\iota}$ that ${\nu^*\alpha\in\H^{1,1}(\tilde{X})\cap\img\tilde{\iota}}$. Using the classical Lefschetz $(1,1)$-Theorem, we can write $\nu^*\alpha=c_1(\fL)$ for a line bundle $\fL$ on $\tilde{X}$. 
	
	Let $F=\nu^{-1}(p)$ be the reduction of any fiber of $\nu$. Using the commutativity of the diagrams
	
	\begin{center}
	\begin{tikzpicture}[scale=1.5]
	\node (A) at (0,0) {$\{p\}$};
	\node (B) at (1,0) {$X$};
	\node (C) at (0,1) {$F$};
	\node (D) at (1,1) {$\tilde{X}$};

	\path[right hook->,font=\scriptsize]
	(A) edge node[above]{$j$} (B)
	(C) edge node[above]{$i$} (D);
	\path[->,font=\scriptsize]
	(C) edge node[left]{$\restr{\nu}{F}$} (A)
	(D) edge node[right]{$\nu$} (B);
	\end{tikzpicture}
	\quad\quad\quad\quad
	\begin{tikzpicture}[scale=1.5]
	\node (A) at (0,0) {$\H^2(\{p\},\,\C)$};
	\node (B) at (2,0) {$\H^2(X,\,\C)$};
	\node (C) at (0,1) {$\H^2(F,\,\C)$};
	\node (D) at (2,1) {$\H^2(\tilde{X},\,\C)$};
	
	\path[->,font=\scriptsize]
	(B) edge node[above]{$j^*$} (A)
	(D) edge node[above]{$i^*$} (C);
	\path[->,font=\scriptsize]
	(A) edge node[left]{$\big(\restr{\nu}{F}\big)^*$} (C)
	(B) edge node[right]{$\nu^*$} (D);
	\end{tikzpicture}
	\end{center}
	we can calculate in $\H^2(F,\,\C)$:
	\[
	c_1(\restr{\fL}{F})=c_1(i^*\fL)=i^*c_1(\fL)=i^*\nu^*\alpha=\big(\restr{\nu}{F}\big)^*j^*\alpha=0,
	\]
	because $j^*\alpha\in\H^2(\{p\},\,\C)=0$.
	
	We have $R^i\nu_*\O_{\tilde{X}}=0$ for all $i>0$, because $X$ has rational singularities. Hence the exponential sequence yields an isomorphism ${c_1\colon R^1\nu_*\O^{\times}_{\tilde{X}}\stackrel{\sim}{\lra}R^2\nu_*\Z}$, induced by the first Chern class. Restricting $\fL$ to analytic neighborhoods $V$ of $F$ gives a germ in the stalks of these sheaves,
	\begin{align*}
	\big(R^1\nu_*\O^{\times}_{\tilde{X}}\big)_p=\lim_V\H^1(V,\,\O_V^{\times})=\lim_V\Pic(V)\ra&\lim_V\H^2(V,\,\Z)=(R^2\nu_*\Z)_p\\
	\big(\restr{\fL}{V}\big)_V\mapsto&\;\big(c_1(\restr{\fL}{V})\big)_V
	\end{align*}
	where the limits are taken over all $V\defeq\nu^{-1}(U)$ for open analytic neighborhoods $U$ of $p$. Theorem~\ref{loj} yields arbitrarily small analytic neighborhoods $U$ of $p$, such that the restrictions $\H^2(\nu^{-1}(U),\,\Z)\stackrel{\sim}{\lra}\H^2(F,\,\Z)$ are isomorphisms. As $c_1\big(\restr{\fL}{F}\big)$ is trivial, the germ on the right hand side is trivial. Since $c_1$ is an isomorphism on germs, the left hand side has to be trivial as well. In other words, there is a $V=\nu^{-1}(U)$ for an analytic neighborhood $U$ of $p$, such that $\restr{\fL}{V}\cong\O_V$ is trivial.
	
	As the fibers are compact, the transition functions of $\fL$ are constant on the fibers. Consequently, the coherent sheaf $\fL'\defeq\nu_*\fL$ is locally free with the same transition functions as $\fL$. As $\nu$ is generically of degree one, we get $\nu^*\nu_*\fL=\fL$. This gives ${\nu^*c_1(\fL')=c_1(\nu^*\nu_*\fL)=\nu^*\alpha}$. Hence ${c_1(\fL')=\alpha}$, because ${\nu^*\colon\H^2(X,\,\C)\ra \H^2(\tilde{X},\,\C)}$ is injective by Theorem~\ref{hdec}.
\end{proof}	

\subsection{Integration of singular top classes}
\label{integration}
We briefly recall how integrating top singular cohomology classes always makes sense on irreducible, compact, complex varieties and show that it is compatible with pullbacks by bimeromorphic maps. We will use this for our application in Section~\ref{hr}, the Hodge Riemann bilinear relations on klt varieties.

First we consider a smooth, compact, $n$-dimensional, complex manifold $X$ with the canonical orientation given by the complex structure. Integrals over classes in $\H^{2n}(X,\,\Z)$ are defined by integrating their de Rham class over $X$. This is a linear map, which by Poincar\'{e} duality is given by taking the cap product with an element $[X]$ of $\H_{2n}(X,\,\Z)$. More precisely, $[X]$ is the fundamental class of $X$ induced by the canonical orientation of $X$, \cite[page~60f]{GH}.

In general, there is no de Rham isomorphism or Poincar\'{e} duality in the singular case. However, these integrals are generalized by ensuring the existence of a canonical fundamental class.

\begin{lem}[Existence of fundamental classes]
	\label{fundamental}
	Let $X$ be an $n$-dimensional, irreducible, compact, complex variety. Then $\H_{2n}(X,\,\Z)$ and $\H^{2n}(X,\,\Z)$ are both isomorphic to $\Z$ and the cap product induces a perfect pairing
	\begin{center}
		\begin{tikzpicture}[scale=1.5]
		\node (A) at (0,0.5){$\cap\colon\H_{2n}(X,\,\Z)\times\H^{2n}(X,\,\Z)$};
		\node (B) at (3,0.5){$\H_0(X,\Z)\cong\Z$};
		\node (D) at (-0.4,0){$(\sigma$};
		\node (E) at (0.15,0){$,$};
		\node (F) at (0.7,0){$\phi)$};
		\node (G) at (3,0){$\sigma\cap\phi$};
		
		\path[->,font=\scriptsize]
		(A) edge (B);
		\path[|->,font=\scriptsize]		
		(F) edge (G);
		\end{tikzpicture}
	\end{center}
\end{lem}
\begin{proof}
	Recall that the underlying topological space of a complex variety admits a triangulation that endows it with the structure of a $2n$-dimensional
	\emph{pseudomanifold}, \cite[1.1]{GM80}. As $X$ is compact, only finitely many simplices are needed. Like in the smooth case,
	$\H_{2n}(X,\,\Z)$ and $\H^{2n}(X,\,\Z)$ of a closed pseudomanifold $X$ are always isomorphic to $\Z$ or $0$, depending on
	the orientability of $X$, \cite[24]{ST80}. Using the complex structure on the smooth locus, all simplices can be canonically oriented \emph{coherently}, such that $X$ is orientable. The class of the sum over all canonically oriented $2n$-simplices generates $H_{2n}(X,\Z)$.
\end{proof}

\begin{notation}
	On an $n$-dimensional, irreducible, compact, complex variety $X$ we notate the generator of $\H_{2n}(X,\Z)$ that is induced by the complex structure as $[X]$ and call it the \emph{canonical fundamental class} of $X$.
\end{notation}

\begin{defi}[Integration of top classes]
	\label{integral}
	Let $X$ be an $n$-dimensional, irreducible, compact, complex variety with canonical fundamental class $[X]$. The integrals over top singular cohomology classes on $X$ are defined as follows.
	\begin{center}
		\begin{tikzpicture}[scale=1.5]
		\node (A) at (0,0.5){$\H^{2n}(X,\,\C)$};
		\node (B) at (2.5,0.5){$\H_0(X,\C)\cong\C$};
		\node (D) at (0,0){$\alpha$};
		\node (E) at (2.5,0){$\int_X\alpha\defeq[X]\cap\alpha$};
		
		\path[->,font=\scriptsize]
		(A) edge node[above]{$\sim$} (B);
		\path[|->,font=\scriptsize]		
		(D) edge (E);
		\end{tikzpicture}
	\end{center}
\end{defi}

\begin{lem}[Pullbacks of integrals]
	\label{pullbackintegrals}
	If $f\colon X\ra Y$ is a bimeromorphic morphism of irreducible, compact, complex varieties, then $\int_Y\alpha=\int_Xf^*\alpha$ for all
	$\alpha\in\H^{2n}(Y,\,\C)$.
\end{lem}
\begin{proof}
	Holomorphic maps induce $\C$-linear maps between the tangent spaces. Therefore $f$ is compatible with the complex orientation, $f_*[X]=[Y]$. We apply the projection formula for the cap product:
	\[
	\int_Y\alpha=[Y]\cap\alpha=f_*[X]\cap\alpha=[X]\cap f^*\alpha=\int_Xf^*\alpha.\qedhere
	\]
\end{proof}

\subsection{Hodge-Riemann bilinear relations}
\label{hr}
On an $n$-dimensional, compact Kähler manifold $(X,\omega)$ we have for any $k\in\{0,\ldots,n\}$ the Hodge-Riemann sesquilinear form
\begin{center}
	\begin{tikzpicture}[scale=1.5]
	\node (A) at (0,0.5){$\psi_{X,\omega}\colon\H^k(X,\,\C)\times\H^k(X,\,\C)$};
	\node (B) at (3,0.5){$\C$};
	\node (C) at (-0.25,0){$(v$};
	\node (D) at (0.35,0){$,$};
	\node (F) at (0.95,0){$w)$};
	\node (G) at (4,0){$(-1)^{\frac{k(k-1)}{2}}\cdot\int_X v\wedge\overline{w}\wedge\omega^{n-k}$};
	
	\path[->,font=\scriptsize]
	(A) edge (B);
	\path[|->,font=\scriptsize]		
	(F) edge (G);
	\end{tikzpicture}
\end{center}
satisfying the property that for any $p,q\in\Nn$ with $p+q=k$ the form $i^{p-q}\cdot\psi_{X,\omega}$ is positive definite on the primitive part $\H^{p,q}(X)_P\defeq\{v\in\H^{p,q}(X)\mid v\wedge\omega^{n-k+1}=0\}$, \cite[Proposition 3.3.15]{Huy05}. In other words,
\[
\forall v\in\H^{p,q}(X)_P\setminus\{0\}\colon i^{p-q}\cdot\psi_{X,\omega}(v,v)>0.
\]
These inequalities are called the Hodge-Riemann bilinear relations. The following Lemma~\ref{localhr} will allow us to deduce an analogue in the singular setting. We will apply this in Corollary~\ref{bilrel} in the case of varieties with at most klt singularities.

\begin{lem}[Local Hodge-Riemann bilinear relations]\label{localhr}
	Let $(X,\omega)$ be a smooth, $n$-dimensional Kähler manifold with a closed $(p,q)$-form $\alpha$ and set $k\defeq p+q$. If $\alpha\wedge\omega^{n-k+1}=0$, then the form $i^{p-q}\cdot(-1)^{k(k-1)/2}\cdot\alpha\wedge\overline{\alpha}\wedge\omega^{n-k}$ is a positive volume form on the subset of $X$ where $\alpha$ does not vanish.
\end{lem}
\begin{proof}
	This can be checked by a calculation in local coordinates at any point $x\in X$, where $\alpha$ does not vanish, which is done in \cite[Corollary~1.2.36]{Huy05}. The case $k\leq2$, which is actually the interesting case for us, can also be found explicitly in \cite[pages~124--125]{GH}.
\end{proof}

%\begin{rem}[Kähler forms and ample divisors] On a complex, projective variety $X$, any ample divisor $A$ induces a Kähler form $\alpha$ on the smooth locus $X_{reg}$ that defines the same class in $\H^2(X,\C)$ as $c_1\big(\O_X(A)\big)$ via the following construction.

%The linear system $|nA|$ of a very ample multiple of $A$ induces an embedding $i:X\ra\P^r$ into a projective space, such that $nA$ is given as the intersection of a hyperplane $H\subset\P^r$ with $X$. The class $c_1\big(\O_{\P^r}(H)\big)\in\H^2(\P^r,\Z)$ is the class of the Fubini study metric $\omega$ on $\P^r$. The restricted form $\alpha:=\frac{1}{n}i^*\omega|_X$ lies in the class ${c_1\big(\O_X(A)\big)\in\H^2(X,\C)}$ and its restriction to $X_{reg}$ is Kähler.
%\end{rem}1

\begin{notation}
	[Hodge-Riemann sesquilinear form]
	For every $n$-dimensional, irreducible, complex projective variety $X$ any ample class ${a\in\H^2(X,\,\C)}$ induces a sesquilinear form $\psi_{X,a}$, defined by
	\begin{center}
		\begin{tikzpicture}[scale=1.5]
		\node (A) at (0,0.5){$\psi_{X,a}\colon\H^k(X,\,\C)\times\H^k(X,\,\C)$};
		\node (B) at (3,0.5){$\C$};
		\node (C) at (-0.25,0){$(\alpha$};
		\node (D) at (0.35,0){$,$};
		\node (F) at (0.95,0){$\beta)$};
		\node (G) at (4,0){$(-1)^{\frac{k(k-1)}{2}}\cdot\int_X \alpha\cup\overline{\beta}\cup a^{n-k}.$};
		
		\path[->,font=\scriptsize]
		(A) edge (B);
		\path[|->,font=\scriptsize]		
		(F) edge (G);
		\end{tikzpicture}
	\end{center}
\end{notation}

\begin{cor}[Singular bilinear relations]
	\label{singularhr}
	Let $X$ be an $n$-dimensional, irreducible, complex projective variety with a resolution $\nu\colon\tilde{X}\ra X$ and an ample class ${a=c_1\big(\O_X(A)\big)\in\H^2(X,\,\C)}$. Then the bilinear relation $i^{p-q}\cdot\psi_{X,a}(\alpha,\alpha)>0$ holds for any class $\alpha\in\H^k(X,\,\C)\setminus\{0\}$ with $\alpha\cup a^{n-k+1}=0$ and ${\nu^*\alpha\in\H^{p,q}(\tilde{X})}$.
	
\end{cor}
\begin{proof}
	By Lemma~\ref{pullbackintegrals} we have
	\begin{align*}
		\psi_{X,a}(\alpha,\alpha)&= (-1)^{\frac{k(k-1)}{2}}\cdot\int_X \alpha\cup\overline{\alpha}\cup a^{n-k}\\
		&=(-1)^{\frac{k(k-1)}{2}}\cdot\int_{\tilde{X}} \nu^*\alpha\cup\overline{\nu^*\alpha}\cup (\nu^*a)^{n-k}.
	\end{align*}
	As $\nu^*A$ is ample away from the exceptional locus $E$ of $\nu$, the class $\nu^*a$ is locally represented by a closed $(1,1)$-form that is Kähler away from the exceptional locus. By Lemma~\ref{localhr}, the class $i^{p-q}\cdot(-1)^{k(k-1)/2}\cdot\nu^*\alpha\cup\overline{\nu^*\alpha}\cup (\nu^*a)^{n-k}$ is away from $E$ and the vanishing locus of $\alpha$ locally represented by a positive volume form. Hence for $\alpha\neq0$ it is positive on a dense open subset, so it has a positive integral. The integral is finite by the definition of $\psi_{X,a}$.
\end{proof}

If we have only a mixed Hodge structure on $\H^k(X,\,\C)$, the condition ${\nu^*\alpha\in\H^{p,q}(\tilde{X})}$ is not convenient. If $X$ has at most klt singularities, we can at least make use of the Hodge decompositions of $\H^1(X,\,\C)$ and $\H^2(X,\,\C)$ from Theorem~\ref{hdec}.

\begin{cor}[Bilinear relations on klt varieties]
	\label{bilrel}
	Let $X$ be an $n$-dimensional, complex projective variety with at most klt singularities and an ample class ${a\in\H^2(X,\,\C)}$. Then the bilinear relation $i^{p-q}\cdot\psi_{X,a}(\alpha,\alpha)>0$ holds for any non-zero $\alpha$ with ${\alpha\cup a^{n-k+1}=0}$ that can be written as the cup-product of classes ${\alpha_1,\ldots,\alpha_r\in\H^1(X,\,\C)\cup\H^2(X,\,\C)}$.
	
	We write $p\defeq\sum p_j$, $q\defeq\sum q_j$ and $k\defeq p+q$ where $(p_j,q_j)$ are such that the $\alpha_j$ lie in the $\H^{p_j,q_j}(X)$-part of the Hodge decompostion of $\H^1(X,\,\C)$ or $\H^2(X,\,\C)$.
\end{cor}
\begin{proof}
	If $\nu\colon\tilde{X}\ra X$ is a resolution of $X$, the class $\nu^*\alpha$ lies in $\H^{p,q}(\tilde{X})$. Thus Corollary~\ref{bilrel} follows from Corollary~\ref{singularhr}.	
\end{proof}

\bibstyle{alpha}
\bibliographystyle{alpha}
\bibliography{general}

\begin{thebibliography}{GKKP11}

\bibitem[BS00]{BS00}
Indranil Biswas and Vasudevan Srinivas.
\newblock A {L}efschetz {$(1,1)$} theorem for normal projective complex
  varieties.
\newblock {\em Duke Math. J.}, 101(3):427--458, 2000.
\newblock
  \href{http://dx.doi.org/10.1215/S0012-7094-00-10132-9}{DOI:10.1215/S0012-7094-00-10132-9},
  \href{https://arxiv.org/abs/math/9904086}{arXiv:9904086}.

\bibitem[Del71]{DeligneHodgeII}
Pierre Deligne.
\newblock Théorie de {H}odge. {II}.
\newblock {\em Inst. Hautes Études Sci. Publ. Math.}, (40):5--57, 1971.
\newblock \href{http://dx.doi.org/10.1007/BF02684692}{DOI:10.1007/BF02684692}.

\bibitem[GH94]{GH}
Phillip Griffiths and Joseph Harris.
\newblock {\em Principles of algebraic geometry}.
\newblock Wiley Classics Library. John Wiley \& Sons Inc., New York, 1994.
\newblock Reprint of the 1978 original,
  \href{http://dx.doi.org/10.1002/9781118032527}{DOI:10.1002/9781118032527}.

\bibitem[GKKP11]{GKKP11}
Daniel Greb, Stefan Kebekus, Sándor~J. Kovács, and Thomas Peternell.
\newblock Differential forms on log canonical spaces.
\newblock {\em Inst. {H}autes {\'E}tudes Sci.~{P}ubl.~{M}ath.}, 114(1):87--169,
  November 2011.
\newblock
  \href{http://dx.doi.org/10.1007/s10240-011-0036-0}{DOI:10.1007/s10240-011-0036-0}
  An extended version with additional graphics is available as
  \href{http://arxiv.org/abs/1003.2913v4}{arXiv:1003.2913v4}.

\bibitem[GM80]{GM80}
Mark Goresky and Robert MacPherson.
\newblock Intersection homology theory.
\newblock {\em Topology}, 19(2):135--162, 1980.
\newblock
  \href{http://dx.doi.org/10.1016/0040-9383(80)90003-8}{DOI:10.1016/0040-9383(80)90003-8}.

\bibitem[GS75]{GS75}
Phillip Griffiths and Wilfried Schmid.
\newblock Recent developments in {H}odge theory: a discussion of techniques and
  results.
\newblock In {\em Discrete subgroups of {L}ie groups and applicatons to moduli
  ({I}nternat. {C}olloq., {B}ombay, 1973)}, pages 31--127. Oxford Univ. Press,
  Bombay, 1975.
\newblock \url{http://publications.ias.edu/node/208}.

\bibitem[Har77]{Har77}
Robin Hartshorne.
\newblock {\em Algebraic geometry}.
\newblock Springer-Verlag, New York, 1977.
\newblock Graduate Texts in Mathematics, No. 52.
  \href{http://dx.doi.org/10.1007/978-1-4757-3849-0}{DOI:10.1007/978-1-4757-3849-0}.

\bibitem[Har80]{MR597077}
Robin Hartshorne.
\newblock Stable reflexive sheaves.
\newblock {\em Math. Ann.}, 254(2):121--176, 1980.
\newblock \href{http://dx.doi.org/10.1007/BF01467074}{DOI: 10.1007/BF01467074}.

\bibitem[Hat02]{Hatcher}
Allen Hatcher.
\newblock {\em Algebraic topology}.
\newblock Cambridge University Press, Cambridge, 2002.
\newblock {A}vailable from the author's home page at
  \url{http://www.math.cornell.edu/~hatcher}.

\bibitem[HM07]{HMcK07}
Christopher~D. Hacon and James Mckernan.
\newblock On {S}hokurov's rational connectedness conjecture.
\newblock {\em Duke Math. J.}, 138(1):119--136, 2007.
\newblock \href{http://dx.doi.org/10.1215/S0012-7094-07-13813-4}
  {DOI:10.1215/S0012-7094-07-13813-4},
  \href{http://arxiv.org/abs/math/0504330v2}{arXiv:0504330v2}.

\bibitem[Huy05]{Huy05}
Daniel Huybrechts.
\newblock {\em Complex geometry}.
\newblock Universitext. Springer-Verlag, Berlin, 2005.
\newblock An introduction.
  {\href{http://dx.doi.org/10.1007/b137952}{DOI:10.1007/b137952}}.

\bibitem[Keb13]{MR3084424}
Stefan Kebekus.
\newblock Pull-back morphisms for reflexive differential forms.
\newblock {\em Adv. Math.}, 245:78--112, 2013.
\newblock
  \href{http://dx.doi.org/10.1016/j.aim.2013.06.013}{DOI:10.1016/j.aim.2013.06.013},
  \href{http://arxiv.org/abs/1210.3255}{arXiv:1210.3255}.

\bibitem[Kir15]{Kir15}
Tim Kirschner.
\newblock {\em Period mappings with applications to symplectic complex spaces},
  volume 2140 of {\em Lecture Notes in Mathematics}.
\newblock Springer, Cham, 2015.
\newblock
  \href{http://dx.doi.org/10.1007/978-3-319-17521-8}{DOI:10.1007/978-3-319-17521-8},
  \href{https://arxiv.org/abs/1210.4197}{arXiv:1210.4197}.

\bibitem[KM98]{KM98}
János Kollár and Shigefumi Mori.
\newblock {\em Birational geometry of algebraic varieties}, volume 134 of {\em
  Cambridge Tracts in Mathematics}.
\newblock Cambridge University Press, Cambridge, 1998.
\newblock
  \href{http://dx.doi.org/10.1017/CBO9780511662560}{DOI:10.1017/CBO9780511662560}.

\bibitem[Kol07]{Kollar07}
János Kollár.
\newblock {\em Lectures on resolution of singularities}, volume 166 of {\em
  Annals of Mathematics Studies}.
\newblock Princeton University Press, Princeton, NJ, 2007.

\bibitem[KP16]{Kebekus2016}
Stefan Kebekus and Thomas Peternell.
\newblock {\em Aspects of the Geometry of Varieties with Canonical
  Singularities}, pages 73--102.
\newblock Springer International Publishing, Cham, 2016.
\newblock
  \href{http://dx.doi.org/10.1007/978-3-319-24460-0_4}{DOI:10.1007/978-3-319-24460-0$\_$4}.

\bibitem[{\L}oj64]{Loj}
Stanis{\l}aw {\L}ojasiewicz.
\newblock Triangulation of semi-analytic sets.
\newblock {\em Ann. Scuola Norm. Sup. Pisa (3)}, 18:449--474, 1964.
\newblock \href{http://eudml.org/doc/83333}{EuDML:83333}.

\bibitem[Nam01]{NamikawaDeformationTheory}
Yoshinori Namikawa.
\newblock Deformation theory of singular symplectic {$n$}-folds.
\newblock {\em Math. Ann.}, 319(3):597--623, 2001.
\newblock \href{http://dx.doi.org/10.1007/PL00004451}{DOI:10.1007/PL00004451},
  \href{http://arxiv.org/abs/math/0010113v2}{arXiv:0010113v2}.

\bibitem[ST80]{ST80}
Herbert Seifert and William Threlfall.
\newblock {\em Seifert and {T}hrelfall: {A} textbook of topology}, volume~89 of
  {\em Pure and Applied Mathematics}.
\newblock Academic Press, Inc. [Harcourt Brace Jovanovich, Publishers], New
  York-London, 1980.
\newblock \url{http://www.maths.ed.ac.uk/~aar/papers/seifthreng.pdf}.

\end{thebibliography}

\end{document}